\documentclass[11pt]{article}

\usepackage{amsmath, amssymb, amsthm, color}

\theoremstyle{plain}
\newtheorem{thm}{Theorem}[section]

\newtheorem*{prop*}{Proposition}
\newtheorem{lem}[thm]{Lemma}
\newtheorem*{lem*}{Lemma}
\newtheorem{dfn}[thm]{Definition}

\newtheorem{ques}[thm]{Question}
\newtheorem*{rem}{Remark}

\newcommand{\vol}{\textrm{Vol}}

\newcommand{\BBE}{\mathbb{E}}

\date{}
\title{\vspace{-0.7cm}A sub-exponential transition of the chromatic generalized Ramsey numbers}

\author{
Choongbum Lee \thanks{Department of Mathematics,
MIT, Cambridge, MA 02139-4307. Email: cb\_lee@math.mit.edu.
Research supported by NSF Grant DMS-1362326.}
\and
Brandon Tran  \thanks{Department of Mathematics,
MIT, Cambridge, MA 02139-4307. Email: btran115@mit.edu.}
}

\begin{document}

\maketitle

\begin{abstract}
A simple graph-product type construction shows that for all natural numbers $r \ge q$,
there exists an edge-coloring of the complete graph on $2^r$ 
vertices using $r$ colors where the graph consisting of the
union of arbitrary $q$ color classes has chromatic number $2^q$.
We show that for each fixed natural number $q$, if there exists an edge-coloring of
the complete graph on $n$ vertices using $r$ colors where the graph consisting 
of the union of arbitrary $q$ color classes has chromatic number at most $2^q -1 $, 
then $n$ must be sub-exponential in $r$.
This answers a question of Conlon, Fox, Lee, and Sudakov.
\end{abstract}

\section{Introduction}

The {\em Ramsey number} of a graph $G$ is defined as the minimum integer $n$ for which
every edge two-coloring of $K_n$, the complete graph on $n$ vertices, admits a 
monochromatic copy of $G$.
Ramsey's theorem asserts that the Ramsey number of the complete graph $K_k$ is finite
for all natural numbers $k$. It is a fundamental result in combinatorics and its
influence extends to various other fields of mathematics. 

Several variants of the Ramsey number has been suggested since its introduction. 
Let $p$ and $q$ be positive integers satisfying $p \ge 2$ and $2 \le q \le {p \choose 2}$.
The {\em generalized Ramsey number} $F(r,p,q)$ is defined
as the minimum $n$ such that for every edge $r$-coloring of $K_n$
there exists a set of $p$ vertices having at most $q-1$ distinct colors on the edges
with both endpoints in the set. 
Note that $F(2,p,2)$ is equivalent to the Ramsey number of $K_p$.
Generalized Ramsey numbers encode several interesting problems in combinatorics in one function, 
and are closely connected to important problems such as 
the Hales-Jewett theorem, the $(6,3)$-problem, quasirandom graphs,  
and Ramsey-coloring of hypergraphs. 
It was introduced by Erd\H{o}s and Shelah \cite{Erdos75, Erdos81}
around 40 years ago and then systematically studied by Erd\H{o}s and Gy\'arf\'as \cite{ErGy}. 

Note that the definition trivially implies $F(r,p,q) \le F(r,p,q')$ for $q' \le q$. 
Erd\H{o}s and Gy\'arf\'as proved a number of interesting results about the function $F(r,p,q)$, demonstrating
how for fixed $p$, the function falls off from being at least exponential in $r$ when $q=2$ 
to being about $\sqrt{2r}$ when $q = {p \choose 2}$ (for $p \ge 4$). 
In the process, they observed that the generalized 
Ramsey numbers satisfy the recurrence relation $F(r,p,q) \le r F(r,p-1,q-1)$ 
for all $r \ge 2$ and $p,q \ge 3$.
To prove the recurrence relation, suppose that $N \ge r F(r,p-1,q-1)$ for some $r\ge 2$, $p,q \ge 3$
and consider an edge $r$-coloring of $K_N$. Fix a vertex $v$, and note that by the
pigeonhole principle, there exists a set $X$ of size at least 
$\lceil\frac{N-1}{r}\rceil \ge F(r,p-1,q-1)$ for which all edges connecting $v$ to $X$
are of the same color. By definition, $X$ contains a set of $p-1$ vertices
having at most $q-2$ colors inside. Together with $v$, this set gives a set of $p$ vertices
having at most $q-1$ colors inside. Therefore $F(r,p,q) \le N = r F(r,p-1,q-1)$. 
Note that since $F(r,2,2)= 2$, this relation implies $F(r,p,p) \le 2r^{q-1}$ for all $p \ge 2$. 
Building on the work of Mubayi \cite{Mubayi}, and Eichhorn and Mubayi \cite{EiMu}, 
recently Conlon, Fox, Lee, and Sudakov \cite{CoFoLeSu1}
showed that $F(r,p,q)$ is super-polynomial in $r$ for all $q \le p-1$.
Hence a super-polynomial to polynomial transition of the generalized Ramsey numbers occurs
at $p = q$. Understanding such transitions is connected to many interesting
problems in Ramsey theory. See \cite{CoFoLeSu} for further information.


In another related work,  Conlon, Fox, Lee, and Sudakov \cite{CoFoLeSu}
introduced the following chromatic number version of generalized Ramsey numbers. 

\begin{dfn}
Let $p$ and $q$ be positive integers satisfying $p \ge 3$ and $2 \le q \le {p \choose 2}$.
For each positive integer $r$, define $F_\chi(r, p, q)$ as the 
minimum integer $n$ for which every edge-coloring of $K_n$ with $r$ colors
contains a $p$-chromatic subgraph receiving 
at most $q-1$ distinct colors on its edges.
\end{dfn}

An edge-coloring of the complete graph is a {\em chromatic-$(p,q)$-coloring} if the 
union of arbitrary $q-1$ color classes has chromatic number at most $p-1$. 
One can alternatively define $F_\chi(r,p,q)$ as the minimum $n$ for which there does not exist a
chromatic-$(p,q)$-coloring of $K_n$. 
While the former definition illuminates the Ramsey-type nature of the function
and the connection between $F(r,p,q)$ and $F_{\chi}(r,p,q)$ more clearly
(for example it immediately implies $F_{\chi}(r,p,q) \le F(r,p,q)$ for all $r,p,q$),
the latter highlights the essence of the function and is arguably a more natural definition.
Conlon, Fox, Lee, and Sudakov realized the importance of chromatic-$(p,q)$-colorings 
while studying a question on generalized Ramsey numbers related to 
the Hales-Jewett theorem. They established some similarities and differences between
$F(r,p,q)$ and $F_{\chi}(r,p,q)$, and suggested to study the two functions in more depth.

Suppose that there exists a chromatic-$(p,q)$-coloring of $K_n$ using $r$ colors. 
Partition the color set into $\lceil \frac{r}{q-1} \rceil$ sets each of size at most $q-1$ and
note that each set induces a graph of chromatic number at most $p-1$. By the product
formula of chromatic numbers, we see that $F_{\chi}(r,p,q) -1 \le (p-1)^{\lceil \frac{r}{q-1} \rceil}$.
Hence we see that the function $F_{\chi}(r,p,q)$ has a natural exponential upper bound. 
Quite surprisingly, this simple bound turns out to be tight for some choice of parameters. 
Consider a complete graph on the vertex set $\{0,1,\ldots, 2^r-1\}$ and color the edge
$\{v,w\}$ with color $i \in [r]$, if the binary expansions of
$v$ and $w$ first differ in the $i$-th digit.
Since each color class induces a bipartite graph, we see that for all $q \le r$, arbitrary
union of $q$ color classes induce a graph of chromatic number at most $2^q$ (one can 
in fact check that the chromatic number is exactly $2^q$). Thus
$F_{\chi}(r,2^{q}+1,q+1) -1 \ge 2^r$, and together with the upper bound established above,
we see that $F_{\chi}(r,2^{q}+1,q+1) = 2^r + 1$ whenever $r$ is divisible by $q$. 
Conlon, Fox, Lee, and Sudakov asked whether these values of $(p,q)$ are the `thresholds'
for $F_\chi(r,p,q)$ being exponential in $r$. More precisely, they 
asked whether $F_{\chi}(r,2^q,q+1) = 2^{o(r)}$ for all $q\ge 2$, and proved that this indeed is the 
case for $q=2$.
In this paper, we positively answer their question.

\begin{thm} \label{thm:main}
$F_{\chi}(r,2^q,q+1) = 2^{o(q)}$ for all $q\ge 2$. 
\end{thm}

In fact, we prove a slightly stronger statement asserting that for all $q \ge 2$, 
there exists a constant $c_q$ such that 
$F_{\chi}(r,2^q,q+1) \le 2^{c_q r^{1-1/q} (\log r)^q}$.
As noted by Conlon, Fox, Lee, and Sudakov,
Theorem~\ref{thm:main} establishes for each fixed $p$, the maximum value of $q$ for which 
$F_{\chi}(r,p,q)$ is exponential in $r$. To see this, suppose that $2^{d-1} < p \le 2^{d}$ for some
natural number $d$. As observed above, we have $F_{\chi}(r,p,d) \ge F_{\chi}(r, 2^{d-1}+1, d) = 2^{r}$.
On the other hand, Theorem~\ref{thm:main} implies $F_{\chi}(r,p,d+1) \le F_{\chi}(r, 2^d, d+1) = 2^{o(r)}$.
Hence $F_{\chi}(r,p,q)$ is exponential in $r$ if and only if $q \le \lceil \log p \rceil$. 

The rest of the paper is organized as follows. In Section~\ref{sec:pre} we prove some lemmas that will 
be repeatedly used throughout the paper. We first prove the $q=3$ case of 
Theorem~\ref{thm:main} in Section~\ref{sec:84} to illustrate the main ideas of our proof, and
then prove the remaining cases in Section~\ref{sec:q}.
We conclude with some remarks in Section~\ref{sec:remark}.

\medskip

\noindent \textbf{Notation}. A graph $G = (V,E)$ is given by a pair of vertex set $V$ and edge set $E$.
The {\em density} of graph is defined as the fraction of pairs of distinct vertices that form an edge, i.e.,
it is $\frac{2|E|}{|V|(|V|-1)}$.
For a family of sets $\mathcal{W}$, define $\vol(\mathcal{W}) := \Big| \bigcup_{W \in \mathcal{W}} W \Big|$.
Hence if $\mathcal{W}$ consists of disjoint sets, then $\vol(\mathcal{W}) = \sum_{W \in \mathcal{W}} |W|$. We use $\log$ to denote natural logarithm.
We use subscripts such as $R_{\ref{lem:intersecting_sets_size}}$ to denote
the constant $R$ from Theorem/Lemma/Proposition~\ref{lem:intersecting_sets_size}.

\section{Preliminaries} \label{sec:pre}

Fix a graph $G$. We say that a pair of vertex subsets $(V_1, V_2)$
is {\em balanced} if $|V_1| = |V_2|$. For a positive real number $\varepsilon$,
a pair of disjoint vertex subsets $(V_1, V_2)$ is \emph{$\varepsilon$-dense}
if for every pair of subsets $U_1 \subseteq V_1$ and $U_2 \subseteq V_2$
satisfying $|U_1| \ge \varepsilon |V_1|$ and $|U_2| \ge \varepsilon |V_2|$, 
we have $e(U_1, U_2) > 0$.
The following lemma asserts that every graph of large density contains a dense pair. 
The logarithmic factor in the exponent can be removed
by using a more detailed analysis such as that used by Peng, R\"odl, and Ruci\'nski \cite{PeRoRu}, 
but we chose to provide a slightly weaker version for the sake of simplicity.

\begin{lem}\label{lem:dense_pairs}
There exists $\varepsilon_0$ such that the following holds for all positive real numbers
$\varepsilon$ and $d$ satisfying $0 < \varepsilon < \varepsilon_0$ and $0 < d <1$.
If $G$ is an $n$-vertex graph of edge density at least $d$, then it contains a
balanced $\varepsilon$-dense pair $(X, Y)$ for which 
\begin{align*}
|X|=|Y|\ge n d^{\log(1/\varepsilon)/\varepsilon }.
\end{align*}
\end{lem}
\begin{proof}
For simplicity, assume that $n$ is even (the odd case can be similarly handled).
Let $G$ be an $n$-vertex graph of density at least $d$. 
Let $V_1 \cup V_2$ be a bipartition of $V(G)$ satisfying $|V_1| = |V_2| = \frac{n}{2}$ chosen
uniformly at random. Then the probability of a fixed edge crossing the partition
is $\frac{2{n-2 \choose n/2 - 1}}{{n \choose n/2}} = \frac{n}{2(n-1)}$. Therefore by 
linearity of expectation, we have $\BBE[e(V_1, V_2)] = \frac{n}{2(n-1)} {n \choose 2}d = \frac{n^2}{4}d$. 
Hence there exists a particular partition $V_1 \cup V_2$ for which $e(V_1, V_2) \ge \frac{n^2}{4}d$. 
Fix the bipartite graph induced on the pair of sets $(V_1, V_2)$. As established above, it has
density at least $d$. 

Let $\rho$ be the minimum positive real number for which there exists a pair $W_1 \subseteq V_1$
and $W_2 \subseteq V_2$ such that $m = |W_1|=|W_2| \ge \rho |V_1|$ and 
$e(W_1,W_2) \ge \rho^{-\varepsilon / \log(1/\varepsilon)} d|W_1||W_2|$. Note that we are taking a minimum over
a non-empty set since the pair $(V_1, V_2)$ satisfies the condition with $\rho = 1$. 
Further, since $e(W_1, W_2) \le |W_1||W_2|$, we have $\rho^{-\varepsilon / \log(1/\varepsilon)}d \le 1$, 
implying $\rho \ge d^{\log(1/\varepsilon)/\varepsilon}$.
Hence the lemma immediately holds if the pair $(W_1, W_2)$ is $\varepsilon$-dense.
Otherwise, there exists a pair of subsets
$(U_{1}, U_{2})$ such that $U_{i} \subseteq W_i$,
$|U_{i}| \ge \varepsilon m$ for both $i=1,2$, and $e(U_1, U_2) = 0$. 
By abusing notation, we take subsets if necessary and assume that
both $U_1$ and $U_2$ have sizes exactly $\lceil \varepsilon m\rceil$. 
Define $\alpha m = \lceil \varepsilon m \rceil$. 
If $\alpha m \ge \frac{m}{3}$, then $\frac{m}{3} \le \alpha m \le \varepsilon m + 1$ implies 
$m \le 3$ if $\varepsilon < \frac{1}{12}$. 
Since $(W_1, W_2)$ is not $\varepsilon$-dense, it is not complete.
Thus we can take a dense pair $(X,Y)$ consisting of a single edge 
satisfying $|X|=|Y|=1 \ge \frac{m}{3} \ge \frac{\rho}{3}|V_1|$ and increase the density
by a multiplicative factor at least $\frac{9}{8}$. 
However, this contradicts the minimality of $\rho$ since
$3^{\varepsilon / \log(1/\varepsilon)} < \frac{9}{8}$ for sufficiently small $\varepsilon$ 
(say $\varepsilon < \frac{1}{12}$).  
Thus we may assume that $\alpha m < \frac{m}{3}$, i.e., $\alpha < \frac{1}{3}$. 
Define $U_1' = W_1 \setminus U_1$ and $U_2' = W_2 \setminus U_2$. 
By the minimality of $\rho$, we see that
\[
	e(U_1', U_2') \le ((1-\alpha)\rho)^{-\varepsilon / \log(1/\varepsilon)}d |U_1'||U_2'|
	\quad \textrm{and} \quad
	e(U_1, U_2') \le (\alpha\rho)^{-\varepsilon  / \log(1/\varepsilon)}d |U_1||U_2'|,
\]
where the second inequality can be obtained by taking the average over all balanced pairs
$(U_1, U_2'')$ with $U_2'' \subseteq U_2'$. Hence
\begin{align*}
	e(W_1, W_2)
	=&\, e(U_1, U_2) + e(U_1', U_2) + e(U_1, U_2') + e(U_1', U_2') \\
	\le&\, 0 + 2\alpha(1-\alpha)m^2 \cdot (\alpha\rho)^{-\varepsilon/\log(1/\varepsilon)}d + (1-\alpha)^2m^2 \cdot ((1-\alpha)\rho)^{-\varepsilon/\log(1/\varepsilon)}d.
\end{align*}
Since $\alpha < \frac{1}{3}$, we have $2\alpha(1-\alpha) < (1-\alpha)^2$.
Further, $\alpha^{-\varepsilon /\log(1/\varepsilon)} + (1-\alpha)^{-\varepsilon/\log(1/\varepsilon)}$ is decreasing in $\alpha$ in the range $\alpha < \frac{1}{3}$.
Thus we may substitute $\alpha = \varepsilon$ to obtain an upper bound on $e(W_1, W_2)$. 
Since $e(W_1, W_2) \ge \rho^{-\varepsilon/\log(1/\varepsilon)} dm^2$, we see that
\begin{align*}
	1
	\le &\,
	2\varepsilon(1-\varepsilon)(\varepsilon)^{-\varepsilon/\log(1/\varepsilon)}
	+ (1-\varepsilon)^2 (1-\varepsilon)^{-\varepsilon/\log(1/\varepsilon)} \\
	= &\, 2\varepsilon(1-\varepsilon) (1 + \varepsilon + O(\varepsilon^2)) + 
	(1 - 2\varepsilon + \varepsilon^2)\left(1 - \frac{\varepsilon^2}{\log(1/\varepsilon)} + O(\varepsilon^{5/2}) \right) \\
	= &\,\left(2\varepsilon + O(\varepsilon^3)\right) + \left(1 - 2\varepsilon + \varepsilon^2 - \frac{\varepsilon^2}{\log (1/\varepsilon)} + O(\varepsilon^{5/2}) \right).
\end{align*}
where the asymptotics is taken as $\varepsilon \rightarrow 0$. The inequality above gives a contradiction for
sufficiently small $\varepsilon$. 
\end{proof}

We also need a technical lemma that will be repeatedly used throughout the rest of this paper. 

\begin{lem}\label{lem:intersecting_sets_size}
There exists a positive real $R$ such that the following holds for all $r\ge R$ and $\varepsilon \le \frac{1}{2}$. If $X$ is a set of size $n$, and $X_1, \ldots , X_r \subset X$ are subsets of size at least $(1-\varepsilon)n$, then there exists a set $I\subseteq [r]$ of size $|I|=\frac{r}{4}$ such that $|\bigcap_{i\in I} X_i| \ge (1-2\varepsilon)^{r/4}n$.
\end{lem}
\begin{proof}
For each $x\in X$, let $d(x)$ denote the number of subsets $X_i$ containing $x$. Then 
\begin{align*}
\sum_{x\in X} d(x) = \sum_i |X_i|\ge r(1-\varepsilon)n
\end{align*}
Now for each set $I\subseteq [r]$ of size $\frac{r}{4}$, let $X_I = \bigcap_{i\in I} X_i$. Then we have
\begin{align*}
\sum_{\substack{I\subseteq [r]\\
			|I|=\frac{r}{4}}} 
		|X_I| = \sum_{x\in X} \binom{d(x)}{r/4} \ge n \binom{r(1-\varepsilon)}{r/4}\ge (1-2\varepsilon)^{r/4}n\binom{r}{r/4},
\end{align*}
where the first inequality holds by convexity since $(1-\varepsilon)r$ is sufficiently large, and the second inequality holds 
since $\binom{r(1-\varepsilon)}{r/4}/ \binom{r}{r/4} \ge \left(\frac{r(1-\varepsilon)-r/4}{3r/4}\right)^{r/4}\ge (1-2\varepsilon)^{r/4}$.
Thus, there exists some set $I$ with $|X_I|\ge (1-2\varepsilon)^{r/4}n$, as desired.
\end{proof}

\section{Upper bound on $F_{\chi}(r,8,4)$} \label{sec:84}

In this section, we prove the first case of our main theorem to illustrate the important ideas in 
a simpler form.
The idea behind many results in Ramsey theory can be summarized as follows: given an $r$-coloring of the complete graph, find a subset of vertices in which less number of colors appear, and then repeat until no colors remain at which point the graph should be empty. For example, the upper bound on $F_\chi(r,4,3)$ of Conlon, Fox, Lee, and Sudakov was based on this idea. By utilizing the concept of $\varepsilon$-dense pairs, they were able to show that the recursion as above occurs. The straightforward generalization of their approach to chromatic-$(8,4)$-coloring fails because the structure one needs in order to force the recursion as above, a well-organized collection of $\varepsilon$-dense pairs, does not necessarily exist. Our key observation is that if such structure does not appear, then we can find a subset of vertices in which a large number of colors are `extremely' sparse. Our proof is based on a modified version of the recursion given above with an extra flexibility that allows us to work with such sparse colors. The following definitions formalize and quantifies the notion of sparseness that we will use.

\begin{dfn}
Let $x,\varepsilon \in [0,1]$ be real numbers and $r_1, r$ be natural numbers. Given an edge-coloring of $K_n$, we say that a color $c$ is \emph{$(x,\varepsilon)$-sparse} if the subgraph consisting of edges of color $c$
has no balanced $\varepsilon$-dense pairs with parts of size at least $xn$. An edge-coloring of $K_n$ is {\em $(r_1, r, x,\varepsilon)$-restricted} if the edge-coloring uses $r$ colors, out of which $r-r_1$ colors are $(x,\varepsilon)$-sparse. 
\end{dfn}

Note that a color is $(0,\varepsilon)$-sparse only if there are no edges of that color.
Define $G(r_1,r,x,\varepsilon)$ as the minimum $n$ such that every $(r_1, r, x,\varepsilon)$-restricted edge-coloring of $K_n$ contains an 8-chromatic subgraph receiving at most 3 distinct colors on its edges.
Note that we can recast the problem of determining $F_{\chi}(r,8,4)$ using this notation since
$F_{\chi}(r,8,4)=G(r,r,0,\varepsilon)$ for any $\varepsilon$. 
We now present the key lemma to our result. It provides a recursive formula for the $G(r_1,r,x,\varepsilon)$ functions described above.

\begin{lem}\label{lem:G_upper_bound}
There exists a constant $R$ such that the following holds. Suppose that $r_1,r$ are non-negative integers satisfying $r_1\ge R$, and $x,\varepsilon$ are positive real numbers satisfying $x\le 1$. 
Then, for every $(r_1, r, x, \varepsilon^2)$-restricted chromatic-$(8,4)$-coloring of $K_n$, either (1) there exists 
a subset of $\beta \alpha_0 n$ vertices on which the coloring is
$\left(\frac{31r_1}{32}, r, \beta^{-1} \alpha_0^{-1}\max(\alpha_1,x),\varepsilon^2\right)$-restricted, or (2) a subset of $\beta \alpha_1 \alpha_0 n$ vertices on which
the coloring is $\left(\frac{31r}{32}, \frac{31r}{32}, 0, \varepsilon^2 \right)$-restricted, where $\alpha_1=e^{-100\log^{2}(r)r^{2/3}}, \alpha_0=r^{-\log(1/\varepsilon^2)/\varepsilon^2},$ and $\beta=(1-16\varepsilon)^{\frac{r}{4}}$.
\end{lem}
\begin{proof}
Let $R = 8 R_{\ref{lem:intersecting_sets_size}}$.
For notational convenience, we will drop $\varepsilon$ from our notation and use $x$-sparse for $(x,\varepsilon^2)$-sparse. 
Suppose that a $(r_1,r,x,\varepsilon^2)$-restricted chromatic-$(8,4)$-coloring of $K_n$ is given. Take a densest color, which we call red, and consider the graph $\mathcal{G}$ induced by its edges. The graph has density $d$ at least $\frac{1}{r}$ and so, by Lemma~\ref{lem:dense_pairs}, we can obtain an $\varepsilon^2$-dense pair $V_1\cup V_{-1}$ with parts of size $m=|V_1|=|V_{-1}|\ge nr^{-\log(1/\varepsilon^2)/\varepsilon^2}=n\alpha_0$.

Define $C_1$ as the set of colors that are not $x$-sparse and $C_2$ as the set of colors that
are $x$-sparse. Hence $|C_1| = r_1$ and $|C_2| = r - r_1$. 
Fix some color $c \in C_1$, and let $\mathcal{G}_c$ be the graph consisting of edges of color $c$. Consider $\mathcal{G}_c[V_1]$, the subgraph of $\mathcal{G}_c$ induced on $V_1$.
Let $W_{1,j} = V_{1,j} \cup V_{1,-j}$ for $j=1,2,\ldots,k_1$, be a maximal collection of vertex-disjoint
$\varepsilon$-dense pairs with parts of size at least $\alpha_1 m$. 
Define $\mathcal{L}_c(V_1) = \bigcup_{j \in [k_1]} \{V_{1,j}, V_{1,-j}\}$.
Similarly, let $W_{-1,j} = V_{-1,j} \cup V_{-1,-j}$ for $j=1,2,\ldots,k_{-1}$, be a maximal collection of 
vertex-disjoint $\varepsilon$-dense pairs in $\mathcal{G}_c[V_{-1}]$ with parts of size at least $\alpha_1 m$. 
Define $\mathcal{L}_c(V_{-1}) = \bigcup_{j \in [k_{-1}]} \{V_{-1,j}, V_{-1,-j}\}$.
We split up the remaining argument into two cases depending on whether
there exists a color $c$ such that $\vol(\mathcal{L}_c(V_{i})) \ge 8\varepsilon m$ for both $i=\pm1$.

\medskip

\noindent {\bf Case 1}. For every color $c \in C_1$, $\textrm{min} \{\vol(\mathcal{L}_c(V_1)), \vol(\mathcal{L}_c(V_{-1}))\} < 8\varepsilon m $. 

The condition of Case I and the maximality of the collection of sets $W_{i,j}$ imply that for each color 
$c \in C_1$, there exists a subset $S_c \subseteq V_1$ 
or $S_c \subseteq V_{-1}$ of size at least
$(1-8\varepsilon)m$ in which there is no $\varepsilon$-dense pair with parts of size 
at least $\alpha_1 m$.
Without loss of generality, we may assume that $S_c \subseteq V_1$ for at least $\frac{r_1}{2}$ colors $c$.
By Lemma \ref{lem:intersecting_sets_size}, there exists a set $\Gamma \subseteq C_1$ of $\frac{r_1}{8}$ colors such that $S = \bigcap_{c \in \Gamma} S_c$ is of size $|S| \ge (1-16\varepsilon)^{\frac{r_1}{8}} m$. 
By abusing notation, we let $\Gamma$ be an arbitrary subset of these colors of size $\frac{r_1}{32}$, and $S$ be an arbitrary subset of the intersection of $S_c$ for these $\frac{r_1}{32}$ colors of size exactly $(1-16\varepsilon)^{\frac{r}{4}} m=\beta m$.

Consider the coloring of the subgraph $K_S = K_n[S]$ induced on $S$. 
Note that all colors in $\Gamma$ have no $\varepsilon$-dense pair of size at least $\alpha_1 m$ in $S$.
Since $|S| = \beta \alpha_0 n$, each color $c \in \Gamma$ is $\beta^{-1}\alpha_{0}^{-1}\alpha_1$-sparse in $K_S$.
Similarly since the colors in $C_1$ were $x$-sparse in $K_n$, they are 
$\beta^{-1}\alpha_0^{-1}x$-sparse in $K_S$.
Hence, the coloring induced on $K_S$ is a $(\frac{31}{32}r_1, r, \beta^{-1}\alpha_0^{-1}\max(\alpha_1, x), \varepsilon^2)$-restricted chromatic-$(8,4)$-coloring of the complete graph on $|S|$ vertices. 

\medskip

\noindent {\bf Case 2}. There exists a color in $C_1$, which we call blue, such that for $i=\pm 1$, $\vol(\mathcal{L}_c(V_i)) \ge 8\varepsilon m$. 

Since the given coloring is a chromatic-$(8,4)$-coloring,
for each color $c \in C_1 \cup C_2$, the graph consisting of edges of colors red, blue, and $c$ is $7$-colorable. Thus, we can partition the vertices into independent sets $A_1, \ldots, A_7$. If some set $A_k$ intersects both $V_1$ and $V_{-1}$ in more than $\varepsilon^2m$ vertices, then by the definition of $\varepsilon^2$-dense pairs, we would have a red edge between a vertex in $A_k\cap V_1$ and a vertex in $A_k\cap V_{-1}$, violating properness of the $7$-coloring. Thus, either $V_1$ or $V_{-1}$, without loss of generality $V_1$, intersects at least $4$ of the $A_i$, say $A_1, A_2, A_3,A_4$, in at most $\varepsilon^2 m$ vertices. 
Let $I_c$ be the indices $i$ for which $|W_{1,i}\cap A_k|\ge \varepsilon |W_{1,i}|$ for some $k\in \{1,2,3,4\}$. Then, we have
\begin{align*}
\varepsilon\left| \bigcup_{i\in I_c} W_{1,i}\right| \le \left| \bigcup_{i\in I_c} (W_{1,i}\cap \bigcup_{1\le k\le 4} A_k) \right| \le  4\varepsilon^2 m.
\end{align*}
We can then deduce that
\begin{align*}
\left| \bigcup_{i\in I_c} W_{1,i}\right|  \le 4\varepsilon m.
\end{align*}
For $i\notin I_c$, consider the $\varepsilon$-dense pair $W_{1,i}=V_{1,i}\cup V_{1,-i}$. For each $k=5,6,7$, by the definition of dense pairs, we have $|V_{1,i}\cap A_k| \le \varepsilon |V_{1,i}|$ or $|V_{1,-i}\cap A_k| \le \varepsilon |V_{1,-i}|$ similarly as above, as otherwise there will be a blue edge within $A_k$. Moreover, for $k=1,2,3,4$, $|V_{1,i}\cap A_k| \le \varepsilon |V_{1,i}|$ and $|V_{1,-i}\cap A_k| \le \varepsilon |V_{1,-i}|$.
Thus for at least one $k=5,6,7$, $|V_{1,i}\cap A_k| \ge (1-6\varepsilon) |V_{1,i}|$ or $|V_{1,-i}\cap A_k| \ge (1-6\varepsilon) |V_{1,-i}|$. 
Hence for all colors $i \notin I_c$, there exists a subset of $V_{1,i}$ of size at least $(1-6\varepsilon)|V_{1,i}|$ 
that contains no edge of color $c$, or a subset of $V_{1,-i}$ of size at least $(1-6\varepsilon)|V_{1,-i}|$ 
that contains no edge of color $c$.
Let $J_c$ be the indices $i$ (positive or negative) for which $V_{1,i}$ contains a subset of size at least 
$(1-6\varepsilon)|V_{1,i}|$ having no edge of color $c$.
For every color $c \in C_1 \cup C_2$, we have
\[
	\left| \bigcup_{i\in J_c} V_{1,i}\right|  
	\ge \frac{1}{2} \left| \bigcup_{i\notin I_c} W_{1,i}\right|  
	\ge \frac{1}{2} \left( \vol(\mathcal{L}_c(V_1)) - 4\varepsilon m \right)
	\quad \text{or} \quad
	\left| \bigcup_{i\in J_c} V_{-1,i}\right| 
	\ge \frac{1}{2} \left( \vol(\mathcal{L}_c(V_{-1})) - 4\varepsilon m \right).
\]
Without loss of generality, we may assume that at least half of the $r$ colors that are not red nor blue
satisfies the former. Let $\mathcal{C}$ be the set of these colors (note that $|\mathcal{C}| \ge \frac{r}{2}$). We then have
\[
	\sum_{c \in \mathcal{C}} \sum_{i \,:\, i \in J_c} |V_{1,i}| 
	\ge \frac{r}{2} \cdot \frac{1}{2}\Big(\vol(\mathcal{L}_c(V_1)) - 4\varepsilon m\Big).
\]
Since
\[
	\sum_{c \in \mathcal{C}} \sum_{i \,:\, i \in J_c} |V_{1,i}|
	=
	\sum_{i} \sum_{c \,:\, i \in J_c} |V_{1,i}|
	\le
	\sum_{i} |V_{1,i}| \cdot \max_{i} |\{c : i \in J_c\}|,
\]
and $\sum_{i} |V_{1,i}| = \vol(\mathcal{L}_c(V_1)) \ge 8\varepsilon m$ we see that
\[
	\max_{i} |\{c : i \in J_c\}| 
	\ge \frac{r}{4} \cdot \frac{\vol(\mathcal{L}_c(V_1)) - 4\varepsilon m}{\vol(\mathcal{L}_c(V_1))}
	\ge \frac{r}{8}.
\]
In particular, there exists an index $\iota$ for which at least $\frac{r}{8}$ colors $c$ satisfy $\iota \in J_c$. Recall that $\iota\in J_c$ implies that there exists a subset of $V_{1,\iota}$ of size at least
$(1-6\varepsilon) |V_{1,\iota}|$ having no edge of color $c$. An application of Lemma \ref{lem:intersecting_sets_size} then gives a set $S$ of size 
$(1-12\varepsilon)^{\frac{r}{32}}|V_{1,\iota}| \ge (1-16\varepsilon)^{\frac{r}{4}}|V_{1,\iota}|$ and a set of $\mathcal{C}'$ of $\frac{r}{32}$ colors such that $S$ contains no edge of color $c$ for any $c\in \mathcal{C}'$.
Thus $S$ is a subset of size at least $\beta |V_{1,\iota}| \ge \beta \alpha_0 \alpha_1 m$ on which the coloring is $\left(\frac{31r}{32}, \frac{31r}{32}, 0, \varepsilon^2 \right)$-restricted.
\end{proof}

\begin{rem}
Note that in Case 2 above, instead of saying that the subgraph induced on $S$
has $\frac{31r}{32}$ colors with no restriction, 
we could have kept track of the previously $x$-sparse colors
to conclude that the subgraph induced on $S$ has some colors that are $\frac{x}{\beta\alpha_1 \alpha_0}$-sparse. However, we would not gain anything from such refined analysis,
since the lemma will later be used in the range $x \ge \alpha_1$. 
Since $\frac{x}{\beta \alpha_1 \alpha_0} \ge 1$ for such values of $x$, there is no advantage in
carrying out this more refined analysis.
\end{rem}

Lastly, we need to take care of the base case of our recursion, for which we have the following lemma:

\begin{lem}\label{lem:G_base_case} For every integer $R$, there exist constants $\gamma$ and $N_0$ such that
if $r$ and $x$ satisfy $x < e^{-(\log(1/\varepsilon)/\varepsilon)\log((r-R)\gamma)}$, then 
\[
	G(R, r, x, \varepsilon) \le N_0.
\]
\end{lem}
\begin{proof}
Define $\gamma = F_\chi(R,8,4)$. Let $N = \gamma(\gamma-1)$ and
consider an $(R,r,x,\varepsilon)$-restricted chromatic-$(8,4)$-coloring of $K_N$.  
Call the $r-R$ colors that are $(x,\varepsilon)$-sparse as {\em restricted}, and
the other $R$ colors as {\em non-restricted}.
If the subgraph $H$ consisting of the edges of non-restricted colors
contains a copy of $K_\gamma$, then by definition, 
we can find $3$ colors whose union has chromatic number $8$. 
Therefore, $H$ does not contain a copy of $K_\gamma$, and thus 
by Tur\'an's theorem, $H$ has density at most $1 - \frac{1}{\gamma}$.
Then the complement of $H$ in $K_N$ has density at least $\frac{1}{\gamma}$. Since $H$
is $(r-R)$-colored with restricted colors, there exists a restricted color of density at least $\frac{1}{(r-R)\gamma}$ in $K_N$.
By Lemma~\ref{lem:dense_pairs}, we can find an $\varepsilon$-dense pair
with parts of size at least $e^{-(\log(1/\varepsilon)/\varepsilon)\log((r-R)\gamma)}N > x N$. However, this contradicts the fact that our color was a restricted color. 
\end{proof}

We can now combine our above results into an upper bound on $F_{\chi}(r,8,4)$. 

\begin{thm}\label{thm:p=3}
There exists a constant $C$ such that the following holds for all $r$:
\[
	F_{\chi}(r,8,4)\le  e^{Cr^{2/3}\log^4(r)}.
\] 
\end{thm}
\begin{proof}
Define $\varepsilon_0 = (\varepsilon_0)_{\ref{lem:dense_pairs}}$ and
let $\varepsilon = r^{-1/3}\log^{2/3}r$.
Let $R > R_{\ref{lem:G_upper_bound}}$ be large enough
so that for all $r \ge R$, we have $\varepsilon < \varepsilon_0$. 
Let $\gamma$ and $N_0$ be the constants from Lemma~\ref{lem:G_base_case} for this value of $R$.
It suffices to prove the theorem for $r \ge R$, since then we can adjust the value of $C$
so that the conclusion holds for all values of $r$. 
Let $\alpha_1=e^{-100\log^{2}(r)r^{2/3}}, \alpha_0=r^{-\log(1/\varepsilon^2)/\varepsilon^2},$ and $\beta=(1-16\varepsilon)^{\frac{r}{4}}$.

Recall that $F_{\chi}(r,8,4)=G(r,r,0,\varepsilon^2)$. 
Define $r_{1,0} = r_{2,0} = r$, $x_{0} = 0$ and $n_0 = G(r,r,0,\varepsilon^2)-1$.
We start with a $(r_{1,0},r_{2,0},x_{0},\varepsilon^2)$-restricted chromatic-$(8,4)$-coloring of a complete
graph on $n_0$ vertices
and repeatedly apply Lemma \ref{lem:G_upper_bound} to reduce the number of colors. 
Suppose that at the $i$-th step we are given a $(r_{1,i},r_{2,i}, x_i, \varepsilon^2)$-restricted coloring on $n_i$ vertices,.
After applying Lemma~\ref{lem:G_upper_bound}, there are two possibilities depending on whether Case 1 or Case 2 of the lemma applies.
If Case $1$ applies, then we obtain a $(\frac{31}{32}r_{1,i}, r_{2,i}, \beta^{-1}\alpha_0^{-1}\max(\alpha_1, x),\varepsilon^2)$-restricted coloring 
on $n_{i+1} = \beta \alpha_0 n$ vertices.
If Case $2$ applies, then we obtain a $(\frac{31}{32}r_{2,i}, \frac{31}{32}r_{2,i}, 0,\varepsilon^2)$-restricted coloring 
on $n_{i+1} = \beta \alpha_1 \alpha_0 n$ vertices.

Repeat the above until the first time $T$ we have $r_{1,T} < R$ and let $x = x_T$.
Note that there are at most $\log_{32/31}(r)$ iterations of Case $2$, and at most $\log_{32/31}(r)$ consecutive iterations of Case $1$. 
It suffices to prove that $x$ is sufficiently small when the process ends, so that the condition of Lemma \ref{lem:G_base_case} is satisfied with $\varepsilon_{\ref{lem:G_base_case}} = \varepsilon^{2}$.
When Case $2$ applies, the density restriction factor $x$ is reset to $0$. So we only need to check the conditions when we have $\log_{32/31}(r)$ consecutive applications of Case $1$. On the first iteration, the density factor is $\beta^{-1}\alpha_0^{-1}\alpha_1$ and each subsequent iteration adds a multiplicative factor of $\beta^{-1}\alpha_0^{-1}$. 
Thus for $c = \frac{1}{\log(32/31)}$ and some positive constant $c'$, we have
\begin{eqnarray*}
x \le 
\alpha_1(\beta^{-1}\alpha_0^{-1})^{c\log(r)}&=& \alpha_1\left(r^{\log(1/\varepsilon^2)/\varepsilon^2}\right)^{c\log(r)} (1-16\varepsilon)^{\frac{-rc\log(r)}{4}}\\
&<&\alpha_1\left(e^{\log(r)\log(1/\varepsilon^2)/\varepsilon^2}\right)^{c\log(r)} e^{-16\varepsilon \frac{-r c\log(r)}{4}}\\
&<& e^{-100\log^2(r) r^{2/3}} e^{c' r^{2/3} \log^{5/3}(r)}\\
&<& e^{-\log(1/\varepsilon^2)/\varepsilon^2 \log(r\gamma)},
\end{eqnarray*}
as desired.

Suppose we applied Case $1$ $t$ times and Case $2$ $s$ times before reaching time $T$. 
We have seen above that $t < (\log_{32/31}(r))^2$ and $s < \log_{32/31}(r)$.
Therefore there exist positive constants $C', C'', C$ such that
\begin{align*}
	F_{\chi}(r,8,4) = G(r,r,0, \varepsilon^2)
	&\le N_0\alpha_0^{-t-s} \beta^{-t-s}\alpha_1^{-s} \\
	&\le N_0e^{C'r^{2/3}\log^{11/3}(r)}e^{C''\log^{3}(r)r^{2/3}}\\
	&\le e^{Cr^{2/3}\log^4(r)}. \qedhere
\end{align*}
\end{proof}

\section{Upper Bound on $F_{\chi}(r, 2^q, q+1)$} \label{sec:q}

Throughout this section, we fix a natural number $q \ge 2$.

\subsection{Well-balanced colors}

In this subsection, we generalize the definitions used in the previous section. 
For technical reasons, we need a slightly different definition of sparsity. 

\begin{dfn}
Let $x,\varepsilon \in [0,1]$ be real numbers. Given an edge-coloring of $K_n$, we say that a color $c$ is \emph{$(x,\varepsilon)$-sparse} if the subgraph consisting of edges of color $c$
has no balanced $\varepsilon$-dense pair with parts of size between $xn$ and $\varepsilon n$. 
\end{dfn}

This definition is not much more restricted than the one given in the previous section.

\begin{lem} \label{lem:mono}
Suppose that a color $c$ is $(x,\varepsilon)$-sparse in an edge coloring of $K_n$. 
Then for every vertex-subset $S$ of size $|S| \ge \gamma n$, the color $c$ is $(\gamma^{-1}x,\varepsilon)$-sparse
in the induced subgraph $K_n[S]$.
\end{lem}
\begin{proof}
In the induced subgraph $K_n[S]$, there are no balanced $\varepsilon$-dense pairs in color $c$ with parts of size between $xn$ and $\varepsilon n$. 
The conclusion follows since $[\gamma^{-1}|S|, \varepsilon|S|] \subseteq [xn, \varepsilon n]$.
\end{proof}

As in the previous section, we will bound $F_{\chi}$ by using a recursive formula.
This time, we keep track of the number of sparse colors with various different sparsity conditions. 

\begin{dfn}
Let $\vec{r}=(r_1, \ldots, r_{q-1})$ be a sequence of non-decreasing non-negative integers  
and $\vec{x}=(x_1, \ldots, x_{q-1})$ be a sequence of non-negative real numbers. 
For a positive real number $\varepsilon$, an edge-coloring of $K_n$ is {\em $(\vec{r},\vec{x}, \varepsilon)$-restricted} if there exist disjoint sets of colors $C_1 \cup \cdots \cup C_{q-1}$ such that for each $i \in [q-1]$, $|C_i|=r_i-r_{i-1}$ (where $r_0 = 0$), and each color in $C_i$ is $(x_i,\varepsilon)$-sparse. 
\end{dfn}

We will always use $x_1 =1$ so that there are $r_1$ colors that are not restricted.
We define $G(\vec{r},\vec{x}, \varepsilon)$ to be the minimum $n$ such that every $(\vec{r},\vec{x}, \varepsilon)$-restricted coloring of $K_n$ contains a $2^q$-chromatic subgraph receiving at most $q$ distinct colors on its edges. Then $F_{\chi}(r,2^q, q+1)=G((r,r,\ldots ,r), (1,0,\ldots,0), \varepsilon)$ for any $\varepsilon$. 

Consider a $(\vec{r},\vec{x}, \varepsilon)$-restricted coloring of $K_n$.
As in the previous section, we take a densest color $c_1$ and
a balanced $\varepsilon$-dense pair $W_1 = V_1\cup V_{-1}$ in the color $c_1$,
where $|V_1| = |V_{-1}| \ge \alpha_0 n$ for some real number $\alpha_0$ to be defined later. 
We refer to $c_1$ as the first level color, and define $\mathcal{L}_1 = \{V_1, V_{-1}\}$.
Suppose that for some $k \in [q-2]$, we are given colors $c_1, \ldots, c_k$ 
 with $k$-th level sets $\mathcal{L}_{k} =\bigcup_{a_1,\ldots, a_{k}} \{V_{a_1, \ldots, a_{k}}, V_{a_1, \ldots, -a_{k}}\}$,
where the sets are paired into $W_{a_1, \ldots, a_{k}} = V_{a_1, \ldots, a_{k}} \cup V_{a_1, \ldots, -a_{k}}$
forming balanced $\varepsilon$-dense pairs in color $c_{k}$.
For a color $c_{k+1}$, we construct the $(k+1)$-th level sets by taking balanced $\varepsilon$-dense pairs of color $c_{k+1}$
in each $V_{\vec{a}} \in \mathcal{L}_{k}$ as follows. 
Take a maximal collection of vertex-disjoint $\varepsilon$-dense pairs $W_{\vec{a}, a_{k+1}} = V_{\vec{a}, a_{k+1}} \cup V_{\vec{a}, -a_{k+1}}$ in $V_{\vec{a}}$ consisting of edges of color $c_{k+1}$
and having sizes $\alpha_k |V_{\vec{a}}| \le |V_{\vec{a}, a_{k+1}}| \le \varepsilon |V_{\vec{a}}|$ for some parameter $\alpha_{k}$ to be defined later. 
Define $\mathcal{L}(V_{\vec{a}}) = \bigcup_{a_{k+1}}\{V_{\vec{a}, a_{k+1}}, V_{\vec{a}, -a_{k+1}}\}$ and
$\mathcal{L}_{k+1} = \bigcup_{V \in \mathcal{L}_{k} } \mathcal{L}(V)$.
Take a subfamily if necessary so that $\vol(\mathcal{L}(V_{\vec{a}})) \le 2^{q+3}\varepsilon |V_{\vec{a}}|$
for each $V_{\vec{a}} \in \mathcal{L}_k$.
Hence
\begin{align} \label{eq:vol_ub}
	\vol(\mathcal{L}_{k+1}) 
	= \sum_{V_{\vec{a}} \in \mathcal{L}_k} \vol(\mathcal{L}(V_{\vec{a}}))
	\le 2^{q+3}\varepsilon \vol(\mathcal{L}_k).
\end{align}
The following definition provides a threshold for a set containing `enough' dense pairs 
in the next level.

\begin{dfn}
For $k \ge 1$, we say that $V \in \mathcal{L}_k$ is {\em properly-shattered} if 
$2^{q+2}\varepsilon|V| \le \vol(\mathcal{L}(V)) \le 2^{q+3}\varepsilon|V|$. 
Let $\mathcal{N}_k \subseteq \mathcal{L}_k$ be the family of sets that are not 
properly-shattered.
\end{dfn}

The following lemma shows that a non-properly-shattered set contains a large subset on
which $c_{k+1}$ is sparse. 

\begin{lem} \label{lem:nonshatter}
If $V \in \mathcal{L}_k$ is not properly-shattered, then there exists $V' \subseteq V$
of size at least $|V'| \ge \Big(1 -  2^{q+2}\varepsilon\Big)|V|$
that contains no balanced $\varepsilon$-dense pair of size between
$\alpha_{k}|V|$ and $\varepsilon |V|$.
\end{lem}
\begin{proof}
Let $\mathcal{W}$ be the maximal family of vertex-disjoint 
balanced $\varepsilon$-dense pairs in $V$ of color $c_{k+1}$ 
with parts of size between $\alpha_k |V|$ and $\varepsilon|V|$.
If $\vol(\mathcal{W}) < 2^{q+2} \varepsilon |V|$, then the conclusion holds
since we can take $V' = V \setminus \bigcup_{W \in \mathcal{W}} W$. 
Otherwise, since $V$ is not properly-shattered, we must have $\vol(\mathcal{W}) > 2^{q+3}\varepsilon |V|$.
In this case $\mathcal{L}(V)$ is obtained by repeatedly removing a balanced dense
pairs from $\mathcal{W}$ until the first time we reach $\vol(\mathcal{W}) < 2^{q+3} \varepsilon |V|$. 
Since each dense pair consists of parts of sizes at most $\varepsilon|V|$, we see that the final family
has volume at least $2^{q+3}\varepsilon|V| - 2\varepsilon|V| > 2^{q+2}\varepsilon|V|$, showing that $V$ is properly-shattered.
\end{proof}

Throughout the process, we will use a different analysis depending on
whether there are enough properly-shattered sets $V \in \mathcal{L}_k$.
This can be considered as the analogue of Cases 1 and 2 of Lemma~\ref{lem:G_upper_bound}.
For technical reasons, we say that $c_1$ is {\em well-balanced}.

\begin{dfn}
For $k \ge 1$, we say that $(c_1, \ldots, c_{k+1})$ is well-balanced if 
$(c_1, \ldots, c_{k})$ is well-balanced, and
$\vol(\mathcal{N}_{k}) \le 2^{-4(q-1)}\vol(\mathcal{L}_{k})$.
Also, we say that an edge-coloring of a complete graph is well-balanced up to the $k$-th level if there exists a sequence of colors $c_1, \ldots, c_k$ such that $(c_1, \ldots, c_k)$ is well-balanced.
\end{dfn}

\subsection{Recursion}

The proof of Theorem~\ref{thm:main} uses a recursive formula obtained by considering the 
maximum level to which the given coloring is well-balanced.
Let $\varepsilon = r^{-1/q} \log r$ and define
\begin{align*}
\delta=e^{-(1/\varepsilon^{q-1})\log{r}},\quad
z=\frac{2^{4q}}{2^{4q}+1},\,\quad
y=\log_{1/z}(r),\quad \textrm{and}\quad
\beta=(1-2^{q+3}\varepsilon)^{\frac{r}{2^{4q}}}.
\end{align*}
For $i =0,1,\ldots, q-2$, define
\begin{equation*}
\gamma_i = \beta \alpha_0 \cdots \alpha_{i} \quad \textrm{and} \quad
\alpha_i = \beta^{-1} (\beta\delta)^{(3y)^i}.
\end{equation*}
Apply the framework
of the previous subsection with this choice of $\{\alpha_i\}_{i=0}^{q-2}$ and $\varepsilon^{q-1}$ instead of $\varepsilon$.

\begin{lem}\label{lem:not_balanced}
There exists $R \in \mathbb{N}$ such that the following holds for all $k < q-1$.
Let $\vec{r}= (r_1, \ldots, r_{q-1})$ be a sequence of non-decreasing non-negative integers satisfying $R \le r_k \le r$ and
let $\vec{x} = (x_1, \ldots, x_{q-1})$ be a sequence of non-negative real numbers.
If a $(\vec{r}, \vec{x}, \varepsilon^{q-1})$-restricted coloring of $K_n$
is well-balanced up to the $k$-th level, but not the $(k+1)$-th level, then
there exists a subset of at least $\gamma_{k-1}n$ vertices on which the coloring is
$(\vec{r}', \vec{x}',\varepsilon^{q-1})$-restricted, where 
\begin{align*}
\vec{r}'&=\left(z r_k , \ldots, z r_k, r_{k+1}, r_{k+2}, \ldots, r_{q-1}\right)\\
\vec{x}'&=\Big(1,0,\ldots , 0, \gamma_{k-1}^{-1}\max(\alpha_k,x_{k+1}) ,\gamma_{k-1}^{-1}x_{k+2}, \ldots , \gamma_{k-1}^{-1}x_{q-1}\Big).
\end{align*}
\end{lem}
\begin{proof}
Let $R = 2^{4q} R_{\ref{lem:intersecting_sets_size}}$.
Denote the set of colors as $C_1\cup \cdots \cup C_{q-1}$, where for each $i$, $|C_i| = r_i - r_{i-1}$ and
each color in $C_i$ is $(x_i, \varepsilon^{q-1})$-sparse. 
By assumption, we have colors $(c_1, \ldots, c_k)$ that are well-balanced but for all colors
$c$, the sequence of colors $(c_1, \ldots, c_k, c)$ is not well-balanced.
Hence if we define $\mathcal{N}_c$ as the family of non-properly-scattered sets $\mathcal{N}_k \subseteq \mathcal{L}_k$
obtained by considering the sequence of colors $(c_1, \ldots, c_{k}, c)$, then
$\vol(\mathcal{N}_c) > 2^{-4(q-1)}\vol(\mathcal{L}_{k})$.
Define $\mathcal{C} = C_1 \cup \cdots \cup C_k$. Then
\[
	\sum_{c \in \mathcal{C}} \vol(\mathcal{N}_c)
	> |\mathcal{C}| \cdot 2^{-4(q-1)}\vol(\mathcal{L}_{k}).
\]
On the other hand,
\[
	\sum_{c \in \mathcal{C}} \vol(\mathcal{N}_c)
	= \sum_{V \in \mathcal{L}_{k}} |V| \cdot |\{c \in \mathcal{C} \,: \, V \in \mathcal{N}_c\} |
	\le \vol(\mathcal{L}_{k}) \cdot \max_{V \in \mathcal{L}_{k}} |\{c \in \mathcal{C} \,: \, V \in \mathcal{N}_c\} |.
\]
Hence there exists $V \in \mathcal{L}_{k}$ such that $V \in \mathcal{N}_c$ for at least 
$2^{-4(q-1)}|\mathcal{C}| = 2^{-4q+4}r_k$ colors $c \in \mathcal{C}$.
Let $\mathcal{C}'$ be an arbitrary subset of $2^{-4q+2}r_k$ of these colors.
Since $V \in \mathcal{L}_k$, we have $|V| \ge \alpha_0 \cdots \alpha_{k-1}n$. 

Note that by Lemma~\ref{lem:nonshatter}, 
for each $c \in \mathcal{C}'$, there exists a subset $V_c \subseteq V$ of size at least
$|V_c| \ge (1 - 2^{q+2}\varepsilon^{q-1})|V|$ that does not contain an 
$\varepsilon^{q-1}$-dense pair of size between $\alpha_k|V|$ and $\varepsilon^{q-1}|V|$.
By Lemma \ref{lem:intersecting_sets_size}, there exists a set $S$ of size $(1-2^{q+3}\varepsilon^{q-1})^{\frac{r_k}{2^{4q}}} |V| \ge \beta |V| \ge \gamma_{k-1} n$ in which $\frac{r_k}{2^{4q}}$ colors do not contain an $\varepsilon^{q-1}$-dense pair of size between $\alpha_k|V|$ and $\varepsilon^{q-1}|V|$. For simplicity we take $\frac{r_k}{2^{4q}+1}$ of these colors, denoting the set as $C^*$.
Consider the coloring of the subgraph $K_S$ induced on $S$. 
Define $C_1' = C_1 \cup \ldots \cup C_k$, $C_2'=\ldots=C_k' = \emptyset$,
$C_{k+1}' = C_{k+1} \cup C^*$, and $C_j' = C_j$ for all $j > k+1$.
Note that in $K_S$, by Lemma~\ref{lem:mono}, the colors in $C_{k+1}'$ are $(\gamma_{k-1}^{-1}\max(\alpha_k, x_{k+1}), \varepsilon^{q-1})$-sparse, and for all $j > k+1$, the colors in $C_{j}'$ are $(\gamma_{k-1}^{-1}x_j, \varepsilon^{q-1})$-sparse.
Further, we have $|C_1'| = zr_k$, $|C_{k+1}'| = r_{k+1} - zr_k$, and $|C_{j}'| = |C_j|$ for all $j > k+1$. Therefore when restricted to $K_S$, the coloring 
is $(\vec{r}', \vec{x}', \varepsilon^{q-1})$-restricted, proving the lemma (where the definitions of $\vec{r}', \vec{x}'$ are given in the statement).
\end{proof}

The following lemma handles the case when the given coloring is well-balanced up to the final, $(q-1)$-th,
level.

\begin{lem}\label{lem:G_balanced}
There exists $R \in \mathbb{N}$ such that the following holds.
Let $\vec{r}= (r_1, \ldots, r_{q-1})$ be a sequence of non-decreasing non-negative integers satisfying $R \le r_{q-1} \le r$ and
let $\vec{x} = (x_1, \ldots, x_{q-1})$ be a sequence of non-negative real numbers.
Suppose that a $(\vec{r}, \vec{x}, \varepsilon^{q-1})$-restricted chromatic-$(2^{q},q+1)$-coloring of $K_n$
is well-balanced up to the $(q-1)$-th level. Then
there exists a subset of at least $\gamma_{q-2}n$ vertices on which the coloring is
$\left(\left(zr_{q-1},\ldots, zr_{q-1}\right),(1,0,\ldots,0),\varepsilon^{q-1}\right)$-restricted.
\end{lem}
\begin{proof}
Let $R = 2^{4q} R_{\ref{lem:intersecting_sets_size}}$.
Denote the set of colors as $\mathcal{C}$, where $|\mathcal{C}| = r_{q-1}$.
Suppose that $(c_1, \ldots, c_{q-1})$ is well-balanced for $c_1,\ldots, c_{q-1} \in \mathcal{C}$.
Since the coloring is a chromatic-$(2^q, q+1)$-coloring, for each color $c$, the graph consisting of edges colored with $c_1,\ldots, c_{q-1}$ and $c$ is $(2^q-1)$-colorable. Thus, we can partition the vertex set into sets $A_1, \ldots, A_{2^q-1}$, each containing no edge of color $c$. Note that each $A_k$ cannot satisfy
$|A_k \cap V_i| \ge \varepsilon^{q-1}|V_i|$ for both $i= \pm 1$ since $V_1 \cup V_{-1}$ is $\varepsilon^{q-1}$-dense. Thus, at least half of them, without loss of generality $A_1, \ldots , A_{2^{q-1}}$, intersect $V_1$ in smaller than $\varepsilon^{q-1}$ fraction of its vertices. 
Define $\mathcal{S}_1=\{V_1\}$. For $m\in [q-1]$, 
we will iteratively construct families $\mathcal{S}_{m} \subseteq \mathcal{L}_m$ so that 
$\vol(\mathcal{S}_m) \ge 2^{-4m+2}\vol(\mathcal{L}_m)$ and 
each set $V \in \mathcal{S}_m$ satisfies $|V \cap A_k| \ge \varepsilon^{q-m}|V|$ for at most 
$2^{q-m}-1$ sets $A_k$. Note that the condition holds for $m=1$.

Suppose that we constructed the family $\mathcal{S}_m$ for some $m \le q-2$.
For each $V \in \mathcal{S}_m$, consider the $(m+1)$-th level sets $V_i, V_{-i} \in \mathcal{L}(V)$
and recall that these sets form $\varepsilon^{q-1}$-dense pairs $W_{i} = V_{i} \cup V_{-i}$.
Let $T$ be the set of indices $k$ for which $|A_k \cap V| \ge \varepsilon^{q-m}|V|$, so that $|T|\le 2^{q-m} - 1$ by construction.
Let $I_m$ be the set of indices $i$ for which $|W_{i} \cap A_k| \ge \frac{1}{2}\varepsilon^{q-m-1} |W_{i}|$ for some $k\notin T$. Then, we have
\begin{equation*}
	\frac{1}{2}\varepsilon^{q-m-1} \left| \bigcup_{i\in I_m} W_{i} \right| \le \left| \bigcup_{i\in I_m} (W_{i} \cap \bigcup_{k\notin T} A_k) \right| \le (2^q-|T|) \varepsilon^{q-m}|V|
\end{equation*}
and deduce that 
\begin{equation*}
\left| \bigcup_{i\in I_m} W_{i} \right| \le (2^{q+1}-2^{q+1-m})\varepsilon|V|.
\end{equation*}

Fix $i\notin I_m$. Since the pair $W_i = V_i \cup V_{-i}$ is $\varepsilon^{q-1}$-dense, each $k\in T$ cannot satisfy both $|V_i \cap A_k| \ge \varepsilon^{q-m-1}|V_i|$ and $|V_{-i} \cap A_k| \ge \varepsilon^{q-m-1}|V_{-i}|$. Moreover, since $i \notin I_m$, each $k \notin T$ satisfies 
$|V_i \cap A_k| \le |W_i \cap A_k| < \frac{1}{2}\varepsilon^{q-m-1}|W_i| = \varepsilon^{q-m-1}|V_i|$.
Similarly, $|V_{-i} \cap A_k| < \varepsilon^{q-m-1}|V_{-i}|$ for each $k \notin T$.
Therefore, at least one of the two sets $V_i$ and $V_{-i}$ intersects at most $\lfloor \frac{|T|}{2} \rfloor \le 2^{q-m-1}-1$ sets $A_k$ in more than $\varepsilon^{q-m-1}$ fraction of its vertices. 
Let $J_m$ be the set of (positive or negative) indices $i$ for which $V_{i}$ satisfies $|V_i \cap A_k| \ge \varepsilon^{q-m-1}|V_i|$ for at most $2^{q-m-1}-1$ sets $A_k$. 
Note that for each $i \notin I_m$, either $i \in J_m$ or $-i \in J_m$. 
Hence
\[
	\sum_{i \in J_m} |V_{i}|
	\ge \frac{1}{2} \left( \vol(\mathcal{L}(V)) - \left|\bigcup_{i \in I_m} W_{i}\right|\right)
	\ge \frac{1}{2} \Big( \vol(\mathcal{L}(V)) - (2^{q+1} - 2^{q+1-m})\varepsilon |V| \Big)
\]
If $V$ is properly-shattered, then $\vol(\mathcal{L}(V)) \ge 2^{q+2} \varepsilon |V|$ and thus
\[
	\sum_{i \in J_m} |V_{i}|
	\ge \frac{1}{4} \vol(\mathcal{L}(V)).
\]

Define $\mathcal{S}_{m+1}$ as the union of the family of sets $V_{i}$ for $i\in J_m$ over all $V \in \mathcal{S}_m$ (here we are abusing notation since the set $J_m$ differs for each $V$).
Recall that $\mathcal{N}_{m} \subseteq \mathcal{L}_m$ is the family of sets that are not properly-scattered.
We have
\begin{align*}
	\vol(\mathcal{S}_{m+1})
	&\,\ge \sum_{V \in \mathcal{S}_m \setminus \mathcal{N}_m } \frac{1}{4} \vol(\mathcal{L}(V)).
\end{align*}
By the definition of properly-scattered sets, well-balanced colors, and our hypothesis $\vol(\mathcal{S}_m) \ge 2^{-4m+2}\vol(\mathcal{L}_m)$,
\begin{align*}
	\sum_{V \in \mathcal{S}_m \setminus \mathcal{N}_m} \vol(\mathcal{L}(V)) 
	\ge&\, \sum_{V \in \mathcal{S}_m \setminus \mathcal{N}_m} 2^{q+2}\varepsilon |V|
	= 2^{q+2}\varepsilon \left( \vol(\mathcal{S}_m) - \vol(\mathcal{N}_m) \right) \\	
	\ge&\,  2^{q+2}\varepsilon \left(2^{-4m+2}\vol(\mathcal{L}_m) - 2^{-4q+4}\vol(\mathcal{L}_m)\right)
	\ge 2^{-4m+1} \cdot 2^{q+2}\varepsilon \vol(\mathcal{L}_m).
\end{align*}
Hence
\[
	\vol(\mathcal{S}_{m+1})
	\ge  2^{-4m-1} \cdot 2^{q+2}\varepsilon \vol(\mathcal{L}_m)
	\ge 2^{-4m-2} \vol(\mathcal{L}_{m+1}),
\]
where the second inequality follows from \eqref{eq:vol_ub}.
Note that $\mathcal{S}_{q-1}$ is the family of sets $V \in \mathcal{L}_{q-1}$ that intersect at most one $A_k$ in more than $\varepsilon$ fraction of its vertices. This implies that there is a subset of size $(1-(2^{q}-2)\varepsilon)|V|$ in $V$ containing no edge of color $c$, obtained by taking $A_k \cap V$. 

Since the analysis above was for a fixed color $c\in \mathcal{C}$, in order to distinguish between different choices of $c$, we abuse notation and write $\mathcal{S}_c$ for the set $\mathcal{S}_{q-1}$
obtained by considering color $c$.
Since $\vol(\mathcal{S}_c) \ge 2^{-4q+6}\vol(\mathcal{L}_{q-1})$ for each $c \in \mathcal{C}$,
we have
\[
	\sum_{c \in \mathcal{C}} \vol(\mathcal{S}_c)
	\ge |\mathcal{C}| \cdot 2^{-4q+6}\vol(\mathcal{L}_{q-1}).
\]
On the other hand,
\[
	\sum_{c \in \mathcal{C}} \vol(\mathcal{S}_c)
	= \sum_{V \in \mathcal{L}_{q-1}} |V| \cdot |\{c \in \mathcal{C} \,: \, V \in \mathcal{S}_c\} |
	\le \vol(\mathcal{L}_{q-1}) \cdot \max_{V \in \mathcal{L}_{q-1}} |\{c \in \mathcal{C} \,: \, V \in \mathcal{S}_c\} |.
\]
Hence there exists $V \in \mathcal{L}_{q-1}$ such that $V \in \mathcal{S}_c$ for at least 
$2^{-4q+6}|\mathcal{C}|$ colors $c \in \mathcal{C}$.
Let $\mathcal{C}'$ be an arbitrary subset of $2^{-4q+2}|\mathcal{C}| = 2^{-4q+2}r_{q-1}$ of these colors.
For each $c \in \mathcal{C}'$, there exists a subset $V_c \subseteq V$ of size at least
$(1-(2^{q}-2)\varepsilon)|V|$ that contains no edge of color $c$. 
Therefore an application of Lemma \ref{lem:intersecting_sets_size} gives a set $S \subseteq V$ of size $(1-(2^{q+1}-4)\varepsilon)^{\frac{r_{q-1}}{2^{4q}}}|V|$ in $V$ and a set $\mathcal{C}''$ of $\frac{r_{q-1}}{2^{4q}}$ colors such that $S$ contains no edge of color $c$ for any $c\in\mathcal{C}''$. For simplicity, we will take an arbitrary subcollection of $\frac{r_{q-1}}{2^{4q}+1}$ of these colors.
Since $S$ has size at least $(1-2^{q+3}\varepsilon)^{\frac{r_{q-1}}{2^{4q}}}|V| \ge \alpha_0 \cdots \alpha_{q-2}\beta n = \gamma_{q-2}^{-1}n$ and the subgraph induced on $S$ is colored by at most $\frac{2^{4q}}{2^{4q}+1}r_{q-1}$ colors, 
this proves the lemma.
\end{proof}

The next lemma takes care of the cases when one of the coordinates of $\vec{r} = (r_1, \ldots, r_{q-1})$
is small, and will be used as the base cases of our recursion.
The proof is almost identical to that of Lemma~\ref{lem:G_base_case}, but needs to be slightly modified
due to the fact that our definition of sparse colors has changed.

\begin{lem}\label{lem:general_base_case}
For every pair of natural numbers $R$ and $k<q-1$, there exist constants $\gamma$ and $N_0$ such that if $\varepsilon, r$ and $x_2, \ldots, x_{q-1}$ satisfy $\max(x_{k+1}, \ldots, x_{q-1})< \varepsilon e^{-(\log(1/\varepsilon)/\varepsilon)\log((r-R)\gamma)}$, then
\begin{equation*}
	G((R, \ldots, R, r_{k+1}, \ldots, r_{q-2}, r), \vec{x}, \varepsilon)
	\le \varepsilon^{-1}N_0.
\end{equation*}
\end{lem}
\begin{proof}
Define $\gamma = F_\chi(R,2^{q},q+1)$. Let $N_0 = \gamma(\gamma-1)$ and
$N$ be a natural number satisfying $N \ge \varepsilon^{-1}N_0$. 
Consider a chromatic-$(2^q,q+1)$-coloring of $K_N$ that is 
$((R, \ldots, R, r_{k+1}, \ldots, r_{q-2}, r), \vec{x}, \varepsilon)$-restricted.
We may view this coloring as a coloring with $r$ colors where
$R$ colors have no restriction, and $r-R$ colors are $(x,\varepsilon)$-sparse
for $x = \max(x_{k+1}, \ldots, x_{q-1})$. We refer to the former $R$ colors
as {\em non-restricted}, and the latter $r-R$ colors as {\em restricted}.
Consider an arbitrary subset of vertices of size $\varepsilon N \ge N_0$
and let $K$ be the subgraph induced on these vertices. 

If the subgraph $H$ of $K$ consisting of the edges of non-restricted colors
contains a copy of $K_\gamma$, then by definition, 
we can find $q$ colors whose union has chromatic number $2^q$. 
Therefore, $H$ does not contain a copy of $K_\gamma$, and thus by Tur\'an's theorem,
$H$ has density at most $1 - \frac{1}{\gamma}$.
Then the complement of $H$ in $K$ has density at least $\frac{1}{\gamma}$. Since $H$
is $(r-R)$-colored with restricted colors, there exists a restricted color of density at least $\frac{1}{(r-R)\gamma}$ in $K_N$.
By Lemma~\ref{lem:dense_pairs}, we can find an $\varepsilon$-dense pair
with parts of size at least $e^{-(\log(1/\varepsilon)/\varepsilon)\log((r-R)\gamma)} \varepsilon N > x N$ and at most $N_0 \le \varepsilon N$. However, this contradicts the fact that our color was a restricted color. \end{proof}

We now combine the results into an upper bound on $F_{\chi}(r, 2^q, q+1)$.

\begin{thm}
\begin{equation*}
F_{\chi}(r, 2^q, q+1) \le e^{Cr^{1-1/q} (\log{r})^q}
\end{equation*}
\end{thm}
\begin{proof}
Define $\varepsilon_0 = (\varepsilon_0)_{\ref{lem:dense_pairs}}$ and recall that $\varepsilon = r^{-1/q} \log r$.
Let $R > \max\{R_{\ref{lem:not_balanced}}, R_{\ref{lem:G_balanced}}\}$ be large enough
so that for all $r \ge R$, we have $\varepsilon < \varepsilon_0$. 
Let $\gamma$ and $N_0$ be the constants from Lemma~\ref{lem:general_base_case} for this value of $R$.
It suffices to prove the theorem for $r \ge R$, since then we can adjust the value of $C$
so that the conclusion holds for all values of $r$. 

Recall that $F_{\chi}(r,2^q,q+1)=G((r,\ldots, r),(1,0,\ldots,0), \varepsilon^{q-1})$.
Define $\vec{r}_0 = (r,\ldots,r)$, $\vec{x}_0 = (1,0,\ldots,0)$, and $n_0 = G(\vec{r}_0,\vec{x}_0,\varepsilon^{q-1})-1$ so that there exists a 
$(\vec{r}_0, \vec{x}_0, \varepsilon^{q-1})$-restricted chromatic-$(2^q, q+1)$-coloring of $K_{n_0}$.
We obtain a bound on $n_0$ by recursively using Lemmas \ref{lem:not_balanced} and \ref{lem:G_balanced}.
 At each step, we take as input a $(\vec{r_i}, \vec{x_i}, \varepsilon^{q-1})$-restricted coloring of the complete graph on $n_i$ vertices and find a $(\vec{r}_{i+1}, \vec{x}_{i+1}, \varepsilon^{q-1})$-restricted coloring of a complete graph on $n_{i+1}$ vertices, using Lemma \ref{lem:not_balanced} or Lemma \ref{lem:G_balanced}.

Given a $(\vec{r_i}, \vec{x_i}, \varepsilon^{q-1})$-restricted coloring of the complete graph on $n_i$ vertices, suppose that it is well-balanced up to the $k$-th level. If $k=q-1$, then by Lemma~\ref{lem:G_balanced}, we may take
\begin{align*}
\vec{r}_{i+1}=(zr_{q-1} ,\ldots, zr_{q-1}),\quad
\vec{x}_{i+1}=(1,0,\ldots,0),\quad\textrm{and} \quad
n_{i+1} \ge \gamma_{q-2} n_{i}.
\end{align*} 
For $k<q-1$, by Lemma~\ref{lem:not_balanced}, we may take
\begin{align*}
\vec{r}_{i+1}&= (zr_{i,k}, \ldots zr_{i,k}, r_{i,k+1} \ldots, r_{i,q-1}), \\
\vec{x}_{i+1}&=(1,0,\ldots,0, \gamma_{k-1}^{-1} \max(\alpha_k, x_{i,k+1}) , \gamma_{k-1}^{-1} x_{i,k+2}, \ldots, \gamma_{k-1}^{-1} x_{i,q-1}), \quad \textrm{and} \\
n_{i+1}&\ge \gamma_{k-1}n_i.
\end{align*}
Repeat the process above as long as $r_{i,k} > R$ for all $k$.
We say that an iteration ran the {\em $k$-th level process} if the coloring
was well-balanced up to the $k$-th level.

Let $T$ be the time of termination.
Recall that $y = \log_{1/z}(r)$ and note that $y \le 2^{4q+1}\log{r}$. 
The termination condition immediately implies that
there can be at most $y$ occurrences of the $(q-1)$-th level process since the $(q-1)$-th coordinate of $\vec{r}_i$ shrinks by a factor of $z$ at each such iteration. 
Similarly, for $k<q-1$, the process terminates if the $k$-th level process occurs more than $y$ times without any occurrence of the $j$-th level process for $j>k$ in-between since the $k$-th co-ordinate of $\vec{r}_i$ increases only if a $j$-th level process for $j>k$ occurs and shrinks by a factor of $z$ at each $k$-th level process. 
We claim that the conditions of Lemma~\ref{lem:general_base_case} are satisfied 
with $\varepsilon_{\ref{lem:general_base_case}} = \varepsilon^{q-1}$ when the process terminates. 
Fix an index $k \in [q-2]$. 
Let $T_0$ be the last iteration before $T$ on which the $j$-th level process
for some $j \ge k+1$ occurred (if there were no such occurrences, then let $T_0 = 0$). 
Since we have $x_{T_0,k+1} = 0$, at the first time $t > T_0$ at which $x_{t,k+1}$ becomes non-zero
(which is when a $k$-th level process occurs),
we have $x_{t,k+1} = \gamma_{k-1}^{-1}\alpha_{k}$. Let $T_1$ be this time.
The observation above implies that for each $j \le k$, there are at most $y^{k-j+1}$ occurrences of
the $j$-th level process from time $T_1$ to $T$. 
Therefore
\begin{align*}
x_{T,k+1}
&\le \alpha_k \gamma_0^{-y^k} \gamma_1^{-y^{k-1}} \cdots \gamma_{k-1}^{-y} \\
&= \alpha_k \alpha_0^{-y^{k} - \ldots - y} \alpha_1^{-y^{k-1}-\ldots -y} \ldots \alpha_{k-1}^{-y} \beta^{-y^k - \ldots -y-1}\\
&\le \alpha_k \alpha_0^{-2y^{k}} \alpha_1^{-2y^{k-1}} \ldots \alpha_{k-1}^{-2y} \beta^{-2y^k}\\
&= (\beta\delta)^{(3y)^k} (\beta\delta)^{-2y^k} (\beta\delta)^{-2\cdot 3y^k} \ldots (\beta\delta)^{-2\cdot 3^{k-1}y^k} \beta^{-2y^k-1+2y^k + 2y^{k-1} + \ldots + 2y}\\
&\le (\beta\delta)^{y^k}\\
&\le ((1-2^{q+3}\varepsilon)^{\frac{r}{2^{4q}}})^{y^k} (e^{-\varepsilon^{1-q}\log(r)})^{y^k}\\
&\le e^{-r^{(q-1)/q} (\log r) 2^{3-3q} y^k}e^{-r^{(q-1)/q} (\log r)^{2-q} y^k}\\
&\le \varepsilon^{q-1} e^{-(\log(1/\varepsilon^{q-1})/ \varepsilon^{q-1} )\log(r\gamma)}.
\end{align*}
Hence the conditions of Lemma~\ref{lem:general_base_case} are satisfied at time $T$
and we have $n_T \le \varepsilon^{-(q-1)}N_0$.
Suppose for each $k$, we applied  the $k$-th level process $a_k$ times before reaching time $T$.
The discussion above implies $a_k \le y^{q - k}$ for all $k \in [q-1]$.
Therefore
\begin{align*}
	n_0
	&\le \varepsilon^{-(q-1)} N_0 \gamma_0^{-a_1} \gamma_1^{-a_2} \cdots \gamma_{q-2}^{-a_{q-1}}\\
	&\le \varepsilon^{-(q-1)}N_0 \gamma_0^{-y^{q-1}} \gamma_1^{-y^{q-2}} \cdots \gamma_{q-2}^{-y}\\
	&= \varepsilon^{-(q-1)}N_0 \alpha_0^{-y^{q-1} - \ldots - y} \alpha_1^{-y^{q-2}-\ldots -y} \ldots \alpha_{q-2}^{-y} \beta^{-y^{q-1} - \ldots -y}\\
	&\le \varepsilon^{-(q-1)}N_0 \alpha_0^{-2y^{q-1}} \ldots \alpha_{q-2}^{-2y} \beta^{-2y^{q-1}}\\
	&= \varepsilon^{-(q-1)}N_0 (\beta\delta)^{-2y^{q-1}} (\beta\delta)^{-2\cdot 3y^{q-1}} \ldots (\beta\delta)^{-2\cdot 3^{q-2}y^{q-1}} \beta^{-2y^{q-1}+2y^{q-1} + 2y^{q-2} + \ldots + 2y}\\
	&\le \varepsilon^{-(q-1)}N_0 (\beta\delta)^{-3^{q-1}y^{q-1}}\\
	&\le \varepsilon^{-(q-1)}N_0 (\beta\delta)^{-3^{q-1}2^{(4q+1)(q-1)}(\log{r})^{q-1}}\\
	&< \varepsilon^{-(q-1)}N_0 e^{C_1\varepsilon r(\log{r})^{q-1}}e^{C_2/\varepsilon^{q-1} (\log{r})^q},
\end{align*}
where $C_1, C_2$ are positive constants.
Since $\varepsilon = r^{-1/q} \log r$, we have
\begin{equation*}
F_{\chi}(r, 2^q, q+1) \le e^{Cr^{1-1/q} (\log{r})^q},
\end{equation*}
for some constant $C$, as desired.
\end{proof}

\section{Remarks} \label{sec:remark}

The study of chromatic generalized Ramsey numbers raises interesting questions regarding the structure
of edge-colorings of complete graphs. These questions seem to be new types of questions
that have not been asked before.
Establishing lower bounds on $F_{\chi}(r,p,q)$ for various choices of 
parameters $(p,q)$ seems especially interesting since these questions ask to
find an edge-coloring of the complete graph where the union of color classes have small chromatic number.
For example, Conlon, Fox, Lee, and Sudakov \cite{CoFoLeSu} found a 
chromatic-$(4,3)$-coloring proving $F_{\chi}(r,4,3) \ge 2^{\Omega(\log^2 r)}$,
and used it to study a problem of Graham, Rothschild, and Spencer related to
the Hales-Jewett theorem.
One can see, using the product formula for chromatic number of union of graphs, that a chromatic-$(4,3)$-coloring is a chromatic $(2^q, q+1)$-coloring for all $q \ge 2$ (for $q=3$ we need the additional condition that 
each color class induces a bipartite graph). 
Hence their result implies 
\[
	F_{\chi}(r,2^q, q+1)
	\ge F_{\chi}(r,4,3) \ge 2^{\Omega(\log^2 r)}.
\]
It would be interesting to improve this bound for $q > 2$. 
The following question posed by Conlon, Fox, Lee, and Sudakov is closely related, and a positive answer to it will establish for each fixed $p$, the maximum value of $q$ 
for which $F_{\chi}(r,p,q)$ is super-polynomial.

\begin{ques}
Is $F_{\chi}(r,p,p-1)$ super-polynomial in $r$ for all $p > 4$?
\end{ques}

As mentioned in the introduction, the corresponding question for generalized Ramsey numbers has been
answered by Conlon, Fox, Lee, and Sudakov \cite{CoFoLeSu1} who provided an explicit edge-coloring
of the complete graph in which all subgraphs induced on $p$-vertex subsets contain at least $p-1$ 
distinct colors.
It is not clear whether the coloring (or some modification of it) can be used to answer the question above. 

\medskip

\noindent \textbf{Acknowledgements}. This work was done as a UROP (Undergraduate Research Opportunity Program) project. We thank Jacob Fox for suggesting the problem. We also thank Rik Sengupta for fruitful discussions.

\end{document}